\documentclass[12pt,a4paper]{article}
\usepackage{amsmath, enumerate, bm, geometry, xspace, color}
\usepackage{amsthm}
\usepackage{amssymb,tikz}
\usepackage{tikz}
\usepackage{graphicx}

\usetikzlibrary{shapes.geometric}

\usepackage{latexsym}
\usepackage{amsmath,amsthm}
\usepackage{epsfig}
\usepackage{epic}
\usepackage{amssymb}
\usepackage{enumitem}
\usepackage{calc}
\usepackage{authblk}
\newtheorem{theorem}{Theorem}[section]
\newtheorem{corollary}[theorem]{Corollary}

\newtheorem{example}[theorem]{Example}

\newtheorem{lemma}[theorem]{Lemma}
\newtheorem{proposition}[theorem]{Proposition}
\newtheorem{definition}[theorem]{Definition}

\newcommand\blfootnote[1]{%
  \begingroup
  \renewcommand\thefootnote{}\footnote{#1}%
  \addtocounter{footnote}{-1}%
  \endgroup
}
\title{Coulson Integral Formula for the Vertex Energy of a Graph}

\author[1]{Octavio Arizmendi}
\author[2]{Beatriz Carely Luna Olivera}
\author[3]{Marcelino Ram\'irez Ib\'a\~nez}

\affil[1]{Centro de Investigaci\'on en Matem\'aticas}
\affil[2]{Universidad del Papaloapan}
\affil[3]{Universidad Pedag\'ogica Nacional}

\begin{document}

\maketitle

\abstract{ In this note we prove that the vertex energy of a graph, as defined in  \cite{ArJu}, can be calculated in terms of a Coulson integral formula. We present examples of how this formula can be used, and we show some applications to bipartite graphs.
}

\blfootnote{\begin{minipage}[l]{0.3\textwidth} \includegraphics[trim=10cm 6cm 10cm 5cm,clip,scale=0.15]{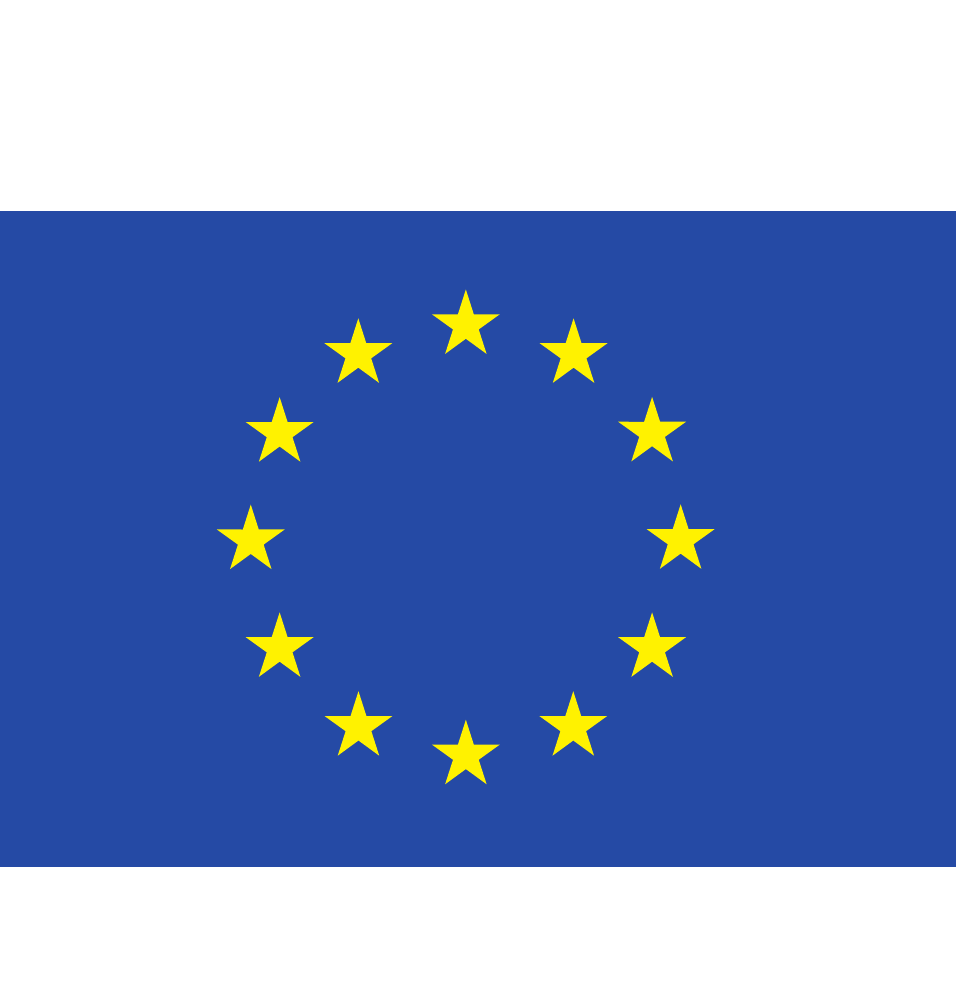} \end{minipage}
 \hspace{-3cm} \begin{minipage}[l][1cm]{0.82\textwidth}
 	 ~\\ This project has received funding from the European Union's Horizon 2020 research and innovation programme under the Marie Sk\l{}odowska-Curie grant agreement No 734922.
 	\end{minipage}}

\section{Introduction}

Let $G$ be a simple graph with $n$ vertices and $m$ edges, and let $A(G)$ be its adjacency matrix  with eigenvalues $\lambda_1\geq\lambda_2\geq\cdots \geq \lambda_n$. We denote by $\phi (G;x) = det(xI_n-A(G))$, the characteristic polynomial of the graph $G$ which is a monic polynomial of degree $n$ whose roots are the eigenvalues of the adjacency matrix $A(G)$. 

  The \emph{energy of a graph $G$} is defined as the sum of the absolute values of these eigenvalues $$\mathcal{E}(G)=\sum^n_{i=1} |\lambda_i|.$$ 
  
    This quantity was introduced by Gutman as a generalization for all graphs of (a linear modification of) the total $\pi$-electron energy of certain conjugated hydrocarbons,  in the H\"uckel
molecular orbital (HMO) approximation. While the original motivation comes from chemistry, we prefer to look at this quantity as a mathematical graph invariant with interesting properties.

   In a fundamental paper \cite{Coul}, Coulson proved the following integral formula for the molecular $\pi$-electron energy, which was easily extended  for the energy of a graph:
\begin{equation} \label{Coulson original}
\mathcal{E}(G) = \frac{1} {\pi }\int\limits_{-\infty }^{+\infty }\left [n -\frac{\mathbf{i}x\,\phi ^{\prime}(G;\mathbf{i}x)} {\phi (G;\mathbf{i}x)} \right ]\mathbf{d}x = \frac{1} {\pi }\int\limits_{-\infty }^{+\infty }\left [n - x \frac{\mathbf{d}} {\mathbf{d}x}\ln \phi (G;\mathbf{i}x)\right ]\mathbf{d}x,
\end{equation}
where $\mathbf{i}=\sqrt{-1}$ and we consider the principal value of the integrals above.

  This formula, known as  Coulson integral formula, has played a very important role in the theory since it allows for the calculation of the energy solely from the characteristic polynomial without necessarily knowing explicitly the eigenvalues of $A(G)$.  In particular, combined with the combinatorial description of the coefficients given by Sachs theorem (see Section 2), it has given a lot of information about the relation between the structural properties of the graph and its energy.

In this paper we are concerned with  the  \emph{energy of a vertex}, as introduced by Arizmendi and Ju\'arez-Romero \cite{ArJu}, see also \cite{AFJ} for further properties.  For a graph $G=(V,E)$, with vertex set $V=\{v_1,v_2,\cdots,v_n\}$,  the energy of the vertex $v_i$ with respect to $G$, denoted by $\mathcal{E}_G(v_i)$, is given by
\begin{equation*}
  \mathcal{E}_G(v_i)=|A|_{ii}, \quad\quad~~~\text{for } i=1,\dots,n,
\end{equation*}
where $|A|=(AA^*)^{1/2}$ and $A=A(G)$ is the adjacency matrix of $G$.

Given $\mathcal{E}(G)=Tr(|A|)$, we can recover the energy of a graph by adding the energies of the vertices in the graph $G$,
\begin{equation*}
  \mathcal{E}(G)=\mathcal{E}_G(v_1)+\cdots+\mathcal{E}_G(v_n).
\end{equation*}

With the Coulson integral formula in mind, it is natural to ask for a similar integral formula for vertices. The main purpose of this note is precisely to derive a refinement of \eqref{Coulson original} to the vertex energy of a graph and to present some applications for its properties.  In the main theorem of this paper, Theorem \ref{Coulson Vertex}, we prove that if $G$ is a graph and $\mathcal{E}_G (v_j)$ is the energy of the vertex $v_j \in G$, then
\begin{equation} \label{eq1}
\mathcal{E}_G(v_j)=\frac{1}{\pi}\int\limits_{-\infty }^{+\infty } 1- \frac{\mathbf{i}x\, \phi(G-v_j;\mathbf{i}x)}{\phi(G;\mathbf{i}x)}dx,
\end{equation}
where $G-v_j$ denotes the graph that results from removing $v_j$ (and the edges containing $v_j$) from $G$. Notice that since $\phi'(G;z)=\phi(G-v_1;z)+\cdots+\phi(G-v_n;z)$, then summing over all the vertices in \eqref{eq1} we obtain Coulson's original  formula \eqref{Coulson original}.

As mentioned above, the importance of Coulson integral formula is that one can get information from it without explicitly calculating the spectrum of the graph and it is especially useful to compare the energy of bipartite graphs by comparing the coefficients of their characteristic polynomials.

In this direction, we are able to use the above formula to compare the energy of two vertices on (possibly different) bipartite graphs and thus formula \eqref{eq1} sheds light on the structural interpretation of the energy of a vertex and its relation to the energy of a graph.

The paper is organized as follows. Section 2 gives the necessary preliminaries. Section 3 is devoted to prove the main theorem. Finally, in section 4, we exploit this to give applications and examples on bipartite graphs by a quasi-order relation.  In particular, this allows us to describe and understand how the energies of the vertices are distributed in the path $P_n$ and some trees $T_n$. We also use it to give inequalities for the energy of coverings and independent sets for these types of graphs.

\section{Preliminaries}

\subsection{Graphs and their spectra}

We will work with simple undirected finite graphs.  A  graph $G$ is a pair $G=(V(G),E(G))$, where $E(G)\subset V(G)\times V(G)$, the elements in $V$ are called vertices and the pairs $(v,w)\in E $ are called edges. Undirected means that $(v,w)\in E$ implies that $(w,v)\in E$ and by simple  we mean that there are no edges of the form $(v,v).$

A finite graph  $G=(V(G),E(G))$ is said to be of order $n$ and size $m$ if $n=|V(G)|$ and $m=|E(G)|$. If $G$ is a graph of order $n$, we label the vertices of the graph $G$ as $v_{1},v_{2}, \ldots, v_{n}$.   The vertices $v_{i}$ and $v_{j}$ are adjacent in $G$ if the edge $(v_{i},v_{j})$ is in $E(G)$; in this case we write $v_i\sim v_j$. The degree of a vertex $v_{i} \in G$ is the number of adjacent vertices to $v_i$ and is denoted by $d_{i}= deg(v_{i})$. A vertex of degree $1$ is called a leaf or pendant vertex.

A path is a sequence of edges $\{e_1,e_2,\dots,e_n\}$ which connects a sequence of vertices $\{v_1,\dots,v_n\}$, i.e.  $e_i=(v_i,v_{i+1})$ and which are all distinct from one another.  A cycle is a closed path, i.e. $v_1=v_n$.
We use  $P_n$ and $C_n$ to denote the path and the cycle of order $n$, respectively. The length of a path (cycle) is the number of edges in the path.
A tree $T$ is a graph that is connected and has no cycles.

In Section 4 we will work with bipartite graphs.  A bipartite graph $G$ is a graph where there are two set of vertices $V_1, V_2\subset V(G)$, called the parts of graph, satisfying the following conditions $(i)$ $V = V_1 \cup V_2$ with $V_1 \cap V_2 = \emptyset$, ($ii$) every edge connects a vertex in $V_1$ with one in  $V_2$.  A graph is bipartite if and only if it has no cycles of odd size. In particular,  trees are biparitite.

The adjacency matrix $\textbf{A}=\textbf{A}(G)$ of $G$ is a square matrix of order $n$ whose $(i,j)$-entry is defined as
\begin{equation}
   A_{ij} = \begin{cases} 1 &\mbox{if } v_{i} \sim v_{j},\\
0 & \mbox{otherwise.} \end{cases}
\end{equation}

The eigenvalues of $A(G)$ are said to be the eigenvalues of the graph $G$. A graph on $n$ vertices has $n$ eigenvalues counted with multiplicity; these will be denoted by $\lambda_{1},\lambda_{2}, \ldots, \lambda_{n}$ and labeled in a decreasing manner: $\lambda_{1} \geq \lambda_{2} \geq \cdots \geq \lambda_{n}$. The set of the all $n$ eigenvalues of $G$ is also called the spectrum of $G$.
Since $G$ is undirected, $A(G)$ is  self-adjoint, and then the eigenvalues of the graph $G$ are necessarily real-valued.
For more details of graph spectrum see \cite{BH} and \cite{CDS}.

The characteristic polynomial of $G$ is given by 
\begin{equation}
\phi(G;x)=det(XI_n-A(G))=\sum_{k=0}^n a_k x^{n-k}
\label{caracteristico}
\end{equation}

Sachs theorem gives a combinatorial way to calculate the coefficients of  the characteristic polynomial.  \begin{theorem}[Sachs theorem]  If the characteristic polynomial of G  is written as  
\begin{equation}
\phi(G;x)=det(XI_n-A(G))=\sum_{k=0}^n a_k x^{n-k}.
\label{sachs}
\end{equation}
Then the coefficients $a_k$ satisfy that  $$a_k=\sum_{S \in L_k} (-1)^{\omega(S)}\,2^{c(S)},$$ where $L_k$ is the set of subgraphs of size $k$ in which every component is $K_2$ or a cycle, $\omega(S)$ is the number of connected components of $S$, $c(S)$  is the number of cycles in $S$, and $a_0=1$. 
\end{theorem}
 
Another approach  is given by Farrell in \cite{FAR}, where he   introduced a class of graph polynomials  in the following way. Let 
 $\mathcal{F}$ be a family of connected graphs, some examples are the families: $\{K_n\},\{C_n\},\{P_n\}, \{S_n\}$, which correspond to complete, cycles, paths and stars, respectively. We  associate a weight $w_\alpha$ to each member $\alpha$ in the family $\mathcal{F}$  we are using. For a graph $G$, an $\mathcal{F}$-{\em subgraph} is a subgraph of $G$ with all its connected components belonging to $\mathcal{F}$. The 
 $\mathcal{F}$-subgraph is said to be {\em proper} if all the components have at least size three. And we said that the $\mathcal{F}$-subgraph is  a $\mathcal{F}$-{\em cover} if it is a spanning subgraph of $G$.
 To the $\mathcal{F}$-subgraph $F$ with connected components $\alpha_1, \ldots, \alpha_k$  we associate 
 the monomial $\Pi(F)=\prod_{i=1}^k w_{\alpha_i}$, with not  necessarily all the $\alpha_i$ different. Now, we can define the 
 $\mathcal{F}$-polynomial of $G$ to be 
 \begin{equation}
 \mathcal{F}(G;{\bf w})=\sum_F \Pi(F),
\end{equation}
where the sum is over all $\mathcal{F}$-subgraphs that are covers of $G$.  Taking the $\{C_n\}$ family, with $C_1=K_1$ and $C_2=K_2$ the improper components, we have the circuit polynomial, $C(G;{\bf w})$. By Sachs Theorem we have a bijection  between $\mathcal{F}$-coverings such that the number of   $C_1$ components is $n-k$ and the $L_k$ Sachs subgraphs, taking  $w_1=x,\, w_2=-1$ and  $w_k=-2$ for $k\geq 3$ we obtain the equation:

 \begin{equation}
   \phi(G;x)=C(G; w_1=x, w_2=-1, w_k=-2\ \text{if} \ k\geq 3).
\end{equation}

\subsection{Energy of vertices}

Let us consider a graph $G=(V,E)$ with vertex set $V=\{v_1,...,v_n\}$ and adjacency matrix $A\in M_n(\mathbb{C})$. If for a matrix $M$, we denote its trace by $Tr(M)$, and its absolute value $(MM^*)^{1/2}$, by $|M|$, then the energy of $G$ is given by
\begin{equation*}
\mathcal{E}(G)=Tr(|A(G)|)=\displaystyle\sum_{i=1}^{n}|A(G)|_{ii}.
\end{equation*}
With this relation in mind, the authors in \cite{ArJu} defined the energy of a vertex as follows.
\begin{definition}
The energy of the vertex $v_i$ with respect to $G$, which is denoted by $\mathcal{E}_G(v_i)$, is given by
\begin{equation}
  \mathcal{E}_G(v_i)=|A(G)|_{ii}, \quad\quad~~~\text{for } i=1,\dots,n,
\end{equation}
where $|A|=(AA^*)^{1/2}$ and $A$ is the adjacency matrix of $G$.
\end{definition}

In this way the energy of a graph is given by the sum of the individual energies of the vertices of $G$,
\begin{equation*}
  \mathcal{E}(G)=\mathcal{E}_G(v_1)+\cdots+\mathcal{E}_G(v_n),
\end{equation*}
and thus the energy of a vertex is a refinement of the energy of a graph.

The following lemma from \cite{AFJ} tells us how to calculate the energy of a vertex in terms of the eigenvalues and eigenvectors of $A$.
\begin{lemma}\label{L1} Let $G=(V,E)$ be a graph with vertices $v_1,...,v_n$. Then
\begin{equation}\label{Equa2}
  \mathcal{E}_{G}(v_i)=\displaystyle\sum_{j=1}^{n}p_{ij}|\lambda_j|,\quad i=1,\ldots,n
\end{equation}
where $\lambda_j$ denotes the $j$-eigenvalue of the adjacency matrix of $A$ and the weights $p_{ij}$ satisfy
$$\sum^n_{i=1}p_{ij}=1\text{ and } \sum^n_{j=1}p_{ij}=1.$$
Moreover, $p_{ij}=u_{ij}^{2}$ where $U = (u_{ij})$ is  the orthogonal matrix whose columns are given by the eigenvectors of $A$.
\end{lemma}

Now since the entry $(i,i)$ of the adjacency matrix of a simple graph equals $0$ then the equation
\begin{equation*}
  0=A(G)_{ii}=\displaystyle\sum_{j=1}^{n}p_{ij}\lambda_j,
\end{equation*}
implies that 
$$-\sum_- \lambda_j p_{ij}=\sum_+ \lambda_j p_{ij}=\frac{1}{2}\mathcal{E}_G(v_i),$$
where the symbol $\sum_-$and $\sum_+$ denote the sum over the negative and positive eigenvalues of the graph $G$, respectively.

 Walks on the graph  will be important to understand the combinatorial properties of the energy.  A walk of length $k$ in $G$ is a sequence of vertices $v_{i_{1}},v_{i_{2}},\cdots ,v_{i_{k}}$ such that $(v_{i_{r}},v_{i_{r+1}})$ is in $E(G)$ for $r=1,2,\ldots,k-1$. We say that a walk is a closed walk if $v_{i_{1}} = v_{i_{k}}$.  A path is a walk where all vertices are distinct.
 A cycle is closed walk with $v_{i_k}\neq v_{i_l}$ for $\leq l,k\leq n$ . If there is a path between  the vertices $v_{i}$  and $v_{j}$ we say that the vertices are connected.   
 
 We will denote by $m_k(G,i)$ the number of $v_i-v_i$ walks in $G$ of length $k$. Notice that $m_k(G,i)$ equals  the quantity $(A^k)_{ii}$ where $A$ is the adjacency matrix of the graph $G$ and coincides with the sum $\sum_j p_{ij}\lambda_j^k$.

\section{A Coulson Integral formula}

In this section we prove the main theorem of the paper which as announced gives a Coulson type integral formula.
Our Coulson type integral formula is based on two representations of a modification of the generating series for the number of walks from $i$ to $i$. 
More precisely, if  we consider $$\Phi_i(z)=\sum_{k=0}^\infty m_k(G,i)z^{-k-1},$$ the following lemma tells us how to write the above series in terms of the weights of  Lemma \ref{L1} and in terms of the characteristic polynomials of $G$ and of $G-v_i$. These facts are already known, but we give a proof for the convenience of the reader.  
\begin{lemma} We have the following representations for $\Phi_i(z)$.
\begin{enumerate} 
\item If $p_{ij}$ are the weights given in \eqref{Equa2} from Lemma \ref{L1}, then
\begin{equation}
\Phi_i(z)=\sum_{j=1}^n \frac{p_{ij}}{z-\lambda_j},
\end{equation}
\item Let $\widetilde{A(G)}_{ii}$ be the matrix obtained from $A(G)$ by eliminating the $i$-th column and the $i$-th row, then
\begin{equation}
\Phi_i(z)=\frac{det(zI-\widetilde{A(G)}_{ii})}{det(zI-A)}=\frac{\phi (G-v_i;z)}{\phi (G;z)}.
\end{equation}  
\end{enumerate} 
\end{lemma}
\begin{proof}
For part 1, note that since $m_k(i)=(A^k)_{ii}$ and $(A^k)_{ii}=(U\Lambda^kU^T)_{ii}$, then $m_k(i)=\sum_{j=1}^n p_{ij}\lambda_j^k$. Now, 
\begin{eqnarray*}
\Phi_i(z) & = &\sum_{k=0}^\infty m_k(i)z^{-k-1}= \sum_{k=0}^\infty \biggl(\sum_{j=1}^n p_{ij}\lambda_j^k \biggr) z^{-k-1}\\
          & = & \sum_{j=1}^n p_{ij} \frac{1}{z}\sum_{k=0}^\infty \left(\frac{\lambda_j}{z}\right)^k 
           =  \sum_{j=1}^n p_{ij}\frac{1}{z}\Bigl(\frac{1}{1-\frac{\lambda_j}{z}}\Bigr)\\
          & = & \sum_{j=1}^n p_{ij}\Bigl(\frac{1}{z-\lambda_j}\Bigl)= \sum_{j=1}^n \frac{p_{ij}}{z-\lambda_j}.       
\end{eqnarray*}
For part 2, if  $A:=A(G)$, we have
\begin{eqnarray*}
\Phi_i(z) & = &\sum_{k=0}^\infty m_k(i)z^{-k-1}= \sum_{k=0}^\infty(A^k)_{ii}z^{-k-1}\\
          & = &\sum_{k=0}^\infty \frac{1}{z}\biggl(\biggl(\frac{A}{z}\biggr)^k\biggr)_{ii}=\frac{1}{z}\Biggl(\sum_{k=0}^\infty \biggl(\frac{A}{z}\biggr)^k\Biggr)_{ii}\\
          & = & \frac{1}{z}\biggl(\frac{1}{1-\frac{A}{z}}\biggr)_{ii}=\biggl(\frac{1}{zI-A}\biggr)_{ii}\\
          & = & (B^{-1})_{ii},
\end{eqnarray*} 
where $B=zI-A$. We know that $B^{-1}=\frac{1}{\text{det}(B)}\textbf{•}{adj}(B)$; for our case, we want the entry $(i,i)$ of this matrix, which is equal to
\begin{eqnarray*}
\bigl(B^{-1}\bigr)_{ii} & = & \frac{1}{\text{det}(B)}(-1)^{i+i}\text{det}\bigl(\widetilde{B}_{ii}\bigr)\\
                        & = & \frac{\text{det}\bigl(z-\widetilde{A}_{ii}\bigr)}{\text{det}(zI-A)}=\frac{\phi(G-v_i;z)}{\phi(G;z)}.
\end{eqnarray*} 
\end{proof}

Now we can prove the main theorem of the paper.
\begin{theorem}[Coulson Integral Formula for Vertices] \label{Coulson Vertex} Let $G$ be a  graph and for a vertex $v_i$ in G, denote by $\mathcal{E}_G (v_i)$ the energy of the vertex $v_i$. Then 
\begin{equation} \label{eq2}
\mathcal{E}_G(v_i)=\frac{1}{\pi}\int_\mathbb{R} 1- \frac{\mathbf{i}x\, \phi(G-v_i;\mathbf{i}x)}{\phi(G;\mathbf{i}x)}dx.
\end{equation}
\end{theorem}

\begin{proof}  The proof follows the same ideas as the original proof of Coulson, 
with the necessary changes using the above considerations. Indeed, by the previous lemma the quotient $\phi(G-v_i;z)/\phi(G;z)$ may be written as $$\frac{\phi(G-v_i;z)}{\phi(G;z)}=\sum^n_{j=1}\frac{p_{ij}}{z-\lambda_j},$$ 
from where, for general $z$, the integrand in the left hand side of \eqref{eq2} may be rewritten as  
$$1-\frac{z\phi(G-v_i;z)}{\phi(G;z)}=1-z\sum_{j=1}^n\frac{p_{ij}}{z-\lambda_j}=\sum_{j=1}^n\frac{\lambda_j p_{ij}}{z-\lambda_j}.$$
Now we take the contour integral along the curve $\Gamma_R$ as shown in the Fig. \ref{contour}. By Cauchy's integral formula we get
\begin{eqnarray*}
\frac{1}{2\pi \mathbf{i}} \oint_{\Gamma_R}  1- \frac{z\phi(G-v_i;z)}{\phi(G;z)}dz&=&\frac{1}{2\pi i} \oint_{\Gamma_R} \sum_{j=1}^n\frac{\lambda_j p_{ij}}{z-\lambda_j}dz
=\frac{1}{2\pi i}  \sum_{j=1}^n \oint_{\Gamma_R}\frac{\lambda_j p_{ij}}{z-\lambda_j}dz\\
&=&\frac{1}{2\pi i}  \sum_- \oint_{\Gamma_R}\frac{\lambda_j p_{ij}}{z-\lambda_j}dz+ \frac{1}{2\pi i}  \sum_+ \oint_{\Gamma_R}\frac{\lambda_j p_{ij}}{z-\lambda_j}dz\\
&=&\sum_-0 +\sum_+ \lambda_j p_{ij}=\sum_+ \lambda_j p_{ij}=\frac{1}{2}\mathcal{E}_G(v_i),
\end{eqnarray*}
where the symbols $\sum_-$ and $\sum_+$ denote, respectively,  the sum over the negative and  positive eigenvalues of the graph $G$.
\begin{figure}[h]
\centering
\includegraphics[height=6cm]{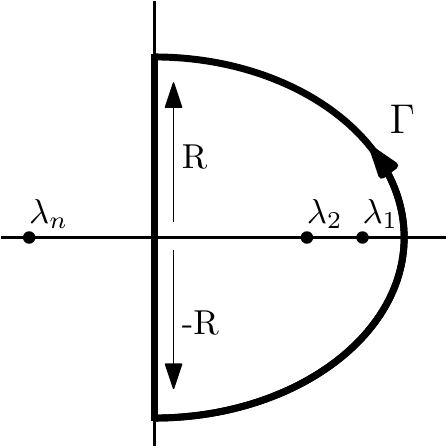}
\caption{The simple positive contour ${\Gamma_R}$ containing the positive roots of $\phi(G;z)$.}
\label{contour}
\end{figure}

The value  of the integral $$\oint_{\Gamma_R}  1-\frac{z\phi(G-v_i;z)}{\phi(G;z)}dz$$ does not change as we change $R$ as long as $R$ is larger than $\lambda_n$ and then 
we may choose an arbitrarily large $R$. Moreover, we may separate the contour integral in two parts, the part along the imaginary axis, denoted by $\Gamma_1$ and the part along the semicircle of radius $R$ with positive real part. 
$$\oint_{\Gamma_R}  1-\frac{z\phi(G-v_i;z)}{\phi(G;z)} dz=\int_{\Gamma_1}  1- \frac{z\phi(G-v_i;z)}{\phi(G;z)}dz+ \int_{\Gamma_2}  1-\frac{z\phi(G-v_i;z)}{\phi(G;z)}.$$

One observes that the second integral vanishes as $n$ goes to infinity. Indeed, the polynomials $\phi(G;z)$ and $z\phi(G-v_i;z)$ are both monic of degree $n$ and their second term is 0, since  $G$ and $G-v_i$ are simple graphs.  Thus, $\phi(G;z)-z\phi(G-v_i;z)$ is a polynomial of degree $n-2$. Thus, if $a$ is its leading coefficient, we see that, as $R \rightarrow \infty$,
\begin{eqnarray}\int_{\Gamma_2}  1- \frac{z\phi(G-v_i;z)}{\phi(G;z)}= \int_{\Gamma_2}  \frac{\phi(G;z)-z\phi(G-v_i;z)}{\phi(G;z)}= \int_{\Gamma_2}  \frac{ax^{n-2}+...}{x^n+...}\\
\leq \frac{a|\Gamma_2|}{R^2}+ o(1/R)=\frac{a\pi}{R}+ o(1/R)\to0 .
\end{eqnarray}
The result now follows in the same way as in Coulson's original proof since, as $R \rightarrow \infty$,
$$\oint_{\Gamma_1}  1- \frac{z\phi(G-v_i;z)}{\phi(G;z)}\to\int^{\infty}_{-\infty}  1- \frac{\mathbf{i}x\phi(G-v_i;\mathbf{i}x)}{\phi(G;\mathbf{i}x)}dx.$$

\end{proof}

We may specialize to pendant vertices obtaining the following expression.

\begin{corollary}Let $uv$ be a pendant edge, with pendant vertex $v$
$$\mathcal{E}_G(v)=\frac{\mathbf{i}}{2\pi}\int_\mathbb{R} \frac{\phi(G-u-v;\mathbf{i}x)(\mathbf{i}x)}{\phi(G;\mathbf{i}x)}dx,$$
\end{corollary}
\begin{proof} It is well known that Sachs theorem implies that $\phi(G)=z\phi(G-v)-\phi(G-u-v)$, thus $\phi(G;\mathbf{i}x)-\mathbf{i}x\phi(G-v;z)=-\phi(G-u-v;\mathbf{i}x)$. 
\end{proof}

Before going into applications of the main theorem to bipartite graphs we provide a pair of examples to show Coulson integral formula for vertices may be used in practice.

\begin{example}[Direct calculation for the star]
The star graph $S_n$ is the graph with vertex set $\{v_1,\dots,v_n\}$ and with edge set $E=\{(v_1,v_j)\ | \ j =2,...,n\}$.  We call $v_1$ the center and $v_2,...v_n$ the leaves.

\begin{figure}[h]
\centering
\begin{tikzpicture}
[mystyle/.style={scale=0.7, draw,shape=circle,fill=black}]
\def\ngon{3}
\def\ngonn{4}
\node[regular polygon,regular polygon sides=\ngon,minimum size=1.5cm] (p) {};
\foreach\x in {1,...,\ngon}{\node[mystyle] (p\x) at (p.corner \x){};}
\node[mystyle] (p0) at (0,0) {};
\foreach\x in {1,...,\ngon}
{
 \draw[thick] (p0) -- (p\x);
}
  \node [label=below:$S_{\ngonn}$] (*) at (0,-0.8) {};
 \end{tikzpicture}
  \qquad
\begin{tikzpicture}
[mystyle/.style={scale=0.7, draw,shape=circle,fill=black}]
\def\ngon{4}
\def\ngonn{5}
\node[regular polygon,regular polygon sides=\ngon,minimum size=1.5cm] (p) {};
\foreach\x in {1,...,\ngon}{\node[mystyle] (p\x) at (p.corner \x){};}
\node[mystyle] (p0) at (0,0) {};
\foreach\x in {1,...,\ngon}
{
 \draw[thick] (p0) -- (p\x);
}
  \node [label=below:$S_{\ngonn}$] (*) at (0,-0.8) {};
 \end{tikzpicture}
  \qquad
\begin{tikzpicture}
[mystyle/.style={scale=0.7, draw,shape=circle,fill=black}]
\def\ngon{5}
\def\ngonn{6}
\node[regular polygon,regular polygon sides=\ngon,minimum size=1.5cm] (p) {};
\foreach\x in {1,...,\ngon}{\node[mystyle] (p\x) at (p.corner \x){};}
\node[mystyle] (p0) at (0,0) {};
\foreach\x in {1,...,\ngon}
{
 \draw[thick] (p0) -- (p\x);
}
  \node [label=below:$S_{\ngonn}$] (*) at (0,-0.8) {};
 \end{tikzpicture}
  \qquad
\begin{tikzpicture}
[mystyle/.style={scale=0.7, draw,shape=circle,fill=black}]
\def\ngon{6}
\def\ngonn{7}
\node[regular polygon,regular polygon sides=\ngon,minimum size=1.5cm] (p) {};
\foreach\x in {1,...,\ngon}{\node[mystyle] (p\x) at (p.corner \x){};}
\node[mystyle] (p0) at (0,0) {};
\foreach\x in {1,...,\ngon}
{
 \draw[thick] (p0) -- (p\x);
}
  \node [label=below:$S_{\ngonn}$] (*) at (0,-0.8) {};
 \end{tikzpicture}
  \qquad
  \begin{tikzpicture}
[mystyle/.style={scale=0.7, draw,shape=circle,fill=black}]
\def\ngon{7}
\def\ngonn{8}
\node[regular polygon,regular polygon sides=\ngon,minimum size=1.5cm] (p) {};
\foreach\x in {1,...,\ngon}{\node[mystyle] (p\x) at (p.corner \x){};}
\node[mystyle] (p0) at (0,0) {};
\foreach\x in {1,...,\ngon}
{
 \draw[thick] (p0) -- (p\x);
}
  \node [label=below:$S_{\ngonn}$] (*) at (0,-0.8) {};
 \end{tikzpicture}
  \caption{Stars $S_4$, $S_5$, $S_6$, $S_7$ and $S_8$ }
\end{figure}
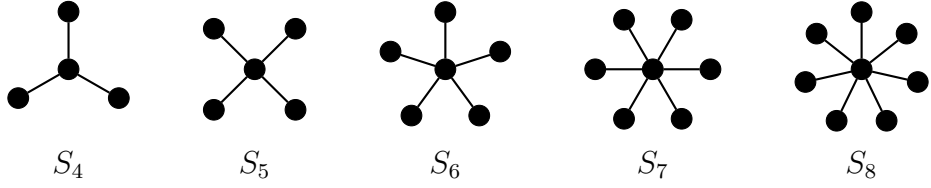

For the center we have
\begin{eqnarray*}\mathcal{E}_G(v)&=&\frac{1}{\pi}\int_\mathbb{R} 1- \frac{(\mathbf{i}x)^n}{(\mathbf{i}x)^n-n (\mathbf{i}x)^{n-2}}dx=\frac{1}{\pi}\int_\mathbb{R}- \frac{n(\mathbf{i}x)^{n-2}}{(\mathbf{i}x)^n-n (\mathbf{i}x)^{n-2}}dx\\&=&\frac{1}{\pi}\int_\mathbb{R} \frac{n}{x^2+n}dx=\frac{\sqrt{n}}{\pi}\int_\mathbb{R} \frac{n}{ny^2+n}dy=\frac{\sqrt{n}}{\pi}\int_\mathbb{R} \frac{1}{y^2+1}dy=\sqrt{n},\end{eqnarray*}
and for the leaves we have a similar calculation
\begin{eqnarray*}\mathcal{E}_G(v)&=&\frac{1}{\pi}\int_\mathbb{R} 1- \frac{(\mathbf{i}x)((\mathbf{i}x)^{n-1}-(n-1) (\mathbf{i}x)^{n-3})}{(\mathbf{i}x)^n-n (\mathbf{i}x)^{n-2}}dx\\
&=&\frac{1}{\pi}\int_\mathbb{R}- \frac{(\mathbf{i}x)^{n-2}}{(\mathbf{i}x)^n-n (\mathbf{i}x)^{n-2}}dx=\frac{1}{\pi}\int_\mathbb{R} \frac{1}{x^2+n}dx=\frac{1}{\sqrt{n}}.
\end{eqnarray*}
\end{example}


\begin{example}[Comparison of two vertices]  \label{comparison}Given the graph $G$ shown in the Fig.\ref{Ex-1} (a), we want to compare the energy of the vertices $v$ and $w$. The characteristic polynomials of the graphs $G$, $G- v$ and $G- w$, shown Fig.\ref{Ex-1}) are 
 \[\phi(G;\mathbf{i}t)=-5t^2-5t^4-t^6,\]
\[\phi(G-v;\mathbf{i}t)=2\mathbf{i}t+4\mathbf{i}t^3+\mathbf{i}t^5,\]
\[\phi(G-w;\mathbf{i}t)=2\mathbf{i}t^3+\mathbf{i}t^5.\]
Notice that for all $t\in \mathbb{R}\setminus\{0\}$, and $\phi(G;\mathbf{i}t)=-5t^2-5t^4-t^6\leq0$,  $\mathbf{i}t\phi(G-v;\mathbf{i}t)> \mathbf{i}t \phi(G-w;\mathbf{i}t)$. So the quotients satisfy
$$\frac{\mathbf{i}x\, \phi(G-v;\mathbf{i}x)}{\phi(G;\mathbf{i}x)}>\frac{\mathbf{i}x\, \phi(G-w;\mathbf{i}x)}{\phi(G;\mathbf{i}x)}$$
and thus  $\mathcal{E}_G(w)> \mathcal{E}_G(v)$.
\end{example}

\begin{figure}\centering
 
  \begin{tikzpicture}

\node[style={circle,fill=black,inner sep=2pt]}] (1) at (0,1) {$.$};
\node[style={circle,fill=black,inner sep=2pt]}] (2) at (1,1) {$.$};
\node[style={circle,fill=black,inner sep=2pt]}] (3) at (2,1) {$.$};
\node[style={circle,fill=black,inner sep=2pt]}] (4) at (3,1) {$.$};
\node[style={circle,fill=black,inner sep=2pt]}] (5) at (4,2) {$.$};
\node[style={circle,fill=black,inner sep=2pt]}] (6) at (4,0) {$.$};
\node[shift={(0.0,1.5)}]{$v$};
 \node[shift={(3.0,1.5)}]{$w$};
\path[draw,thick](1) edge node {} (2);
\path[draw,thick](2) edge node {} (3);
\path[draw,thick](3) edge node {} (4);
\path[draw,thick](4) edge node {} (5);
\path[draw,thick](4) edge node {} (6);
 \node[shift={(2.0,-.5)}]{$(a)$};
 \end{tikzpicture}\qquad\qquad
  \begin{tikzpicture}
\node[style={circle,fill=black,inner sep=2pt]}] (2) at (0,1) {$.$};
\node[style={circle,fill=black,inner sep=2pt]}] (3) at (1,1) {$.$};
\node[style={circle,fill=black,inner sep=2pt]}] (4) at (2,1) {$.$};
\node[style={circle,fill=black,inner sep=2pt]}] (5) at (3,2) {$.$};
\node[style={circle,fill=black,inner sep=2pt]}] (6) at (3,0) {$.$};
\node[shift={(1,2)}]{$G-v$};
\path[draw,thick](2) edge node {} (3);
\path[draw,thick](3) edge node {} (4);
\path[draw,thick](4) edge node {} (5);
\path[draw,thick](4) edge node {} (6);

\node[style={circle,fill=black,inner sep=2pt]}] (1) at (4,1) {$.$};
\node[style={circle,fill=black,inner sep=2pt]}] (2) at (5,1) {$.$};
\node[style={circle,fill=black,inner sep=2pt]}] (3) at (6,1) {$.$};
\node[style={circle,fill=black,inner sep=2pt]}] (5) at (7,2) {$.$};
\node[style={circle,fill=black,inner sep=2pt]}] (6) at (7,0) {$.$};
\node[shift={(5,2)}]{$G-w$};
\path[draw,thick](1) edge node {} (2);
\path[draw,thick](2) edge node {} (3);
 \node[shift={(4.0,-.5)}]{$(b)$};
 \end{tikzpicture}
\caption{ \label{Ex-1}(a) Graph from Example \ref{comparison}, with compared vertices $v$ and $w$. (b) Graphs with deleted vertices. }
\end{figure}
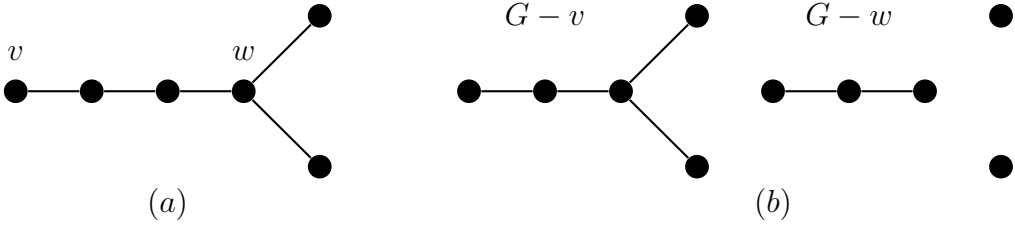

\section{Applications for bipartite graphs}

In this section we consider bipartite graphs. We will use the technique of quasi-order to compare energies of different vertices.

More precisely, recall that, if $G$ is a bipartite graph,  its characteristic polynomial is of the form 
\begin{equation}
\sum_{k\geq 0} (-1)^k b_{2k} x^{n-2k}
\end{equation}
where $b_{2k}\geq 0$ for all $k$. Moreover, for two bipartite graphs one may define the a quasi-order $\preceq$ as follows: $G_1 \preceq G_2$ if $b_{2k}(G_1)\leq b_{2k} (G_2)$ for all $k$, see \cite{GZ,GZ2,Zh,ZL}. This quasi-order has been used to compare graph energies, we will use it to compare vertex energies.

The main result comes from the following lemma which allows us to compare the energy of two vertices in a graph.
\begin{lemma}\label{acyclic}
Let $G$ be a bipartite graph and $v,w\in V(G)$. If $G-w\succcurlyeq G-v$ then $\mathcal{E}_G(w)\leq\mathcal{E}_G(v)$, Moreover if $G-w\neq G-v$ then $\mathcal{E}_G(w)<\mathcal{E}_G(v)$.
\end{lemma}
\begin{proof}
 Since $G$, $G-w$ and $G-v$ are bipartites, we can write
 
\[\phi(G;x)=  \sum_{k\geq 0} (-1)^k a_{2k} x^{n-2k},\]  
\[\phi(G-w;x)=  \sum_{k\geq 0} (-1)^k b_{2k} x^{n-2k},\] 
\[ \phi(G-v;x)=  \sum_{k\geq 0} (-1)^k c_{2k} x^{n-2k}\]

By hypothesis $G-w\succcurlyeq G-v$, which implies $b_{2k}\geq c_{2k}$ for all $k$, 

Now observe that
\begin{eqnarray*}
\mathcal{E}_G(w)&=&\frac{1}{\pi}\int_\mathbb{R} 1- \frac{\mathbf{i}x\, \phi(G-w;\mathbf{i}x)}{\phi(G;\mathbf{i}x)}dx\\
&=&\frac{1}{\pi}\int_\mathbb{R} 1- \frac{\mathbf{i}x (\mathbf{i}x)^{n-1} \, \sum_{k\geq 0} (-1)^k b_{2k} (\mathbf{i}x)^{-2k}}{(\mathbf{i}x)^n\sum_{k\geq 0} (-1)^k a_{2k} (\mathbf{i}x)^{-2k}}dx\\
&=&\frac{1}{\pi}\int_\mathbb{R} 1- \frac{ \sum_{k\geq 0}  (-1)^{k}(\frac{1}{\mathbf{i}^2})^{k} b_{2k} (x)^{-2k}}{\sum_{k\geq 0} (-1)^{k} (\frac{1}{\mathbf{i}^2})^{k}a_{2k} (x)^{-2k}}dx\\
&=&\frac{1}{\pi}\int_\mathbb{R} 1- \frac{ \sum_{k\geq 0}  b_{2k} (x)^{-2k}}{\sum_{k\geq 0} a_{2k} (x)^{-2k}}dx\\
&\leq&\frac{1}{\pi}\int_\mathbb{R} 1- \frac{ \sum_{k\geq 0}  c_{2k} (x)^{-2k}}{\sum_{k\geq 0} a_{2k} (x)^{-2k}}dx\\
&=&\mathcal{E}_G(v)
\end{eqnarray*}
as desired. Notice that equality can only hold if the second to last equation is an equality, which would imply that $G-w$ and $G-v$ are isomorphic.
\end{proof}

By considering the disjoint  union of two graphs we may also compare the energy of  vertices in different graphs. 
\begin{theorem}\label{acyclic2} 
Let $G_1$ and $G_2$ two disjoint bipartite graphs with $w \in G_2$ and $v \in G_1$. If $G_1 \cup (G_2-w) \succeq  G_2 \cup (G_1-v) $, then $\mathcal{E}_{G_2}(w) \leq \mathcal{E}_{G_1}(v)$.
\end{theorem}

\begin{proof}
Since $G_1 \cup (G_2-w)= (G_1 \cup G_2)-w $ and $G_2 \cup (G_1-v)= (G_2 \cup G_1)-v $, taking the graph $G=G_2 \cup G_1$ in the previous lemma the result follows we have that  $\mathcal{E}_{G_1 \cup G_2}(w) \leq \mathcal{E}_{G_1 \cup G_2}(v)$. Finally observe that, since $G_1$ and $G_2$ are disjoint,  $\mathcal{E}_{G_1 \cup G_2}(w)=\mathcal{E}_{ G_2}(w)   $ and   $\mathcal{E}_{G_1 \cup G_2}(v)= \mathcal{E}_{G_2}(v)$.
\end{proof}

The following is a direct consequence of the above theorems.

\begin{corollary}  If $G$ is a tree and $uv$ a pendant edge with pendant vertex $v$, then $v$ has smaller energy than $u$. 
\end{corollary}

\begin{proof}
Notice that $G-u$ consists of a subtree of $G-w$ together with and isolated point, then by Sachs theorem $G-u\succeq G-w$. 
\end{proof}

\begin{proposition} Let $G$ be tree, adding a pendant edge $uv$ at $u$  increases the energy of $u$.  
\end{proposition}

\begin{proof}
Let $P$ and $\tilde P$ the characteristic polynomials of $G$ and $\tilde G=G\cup\{uv\}$, the energies of $u$ in $G$ and  $\tilde G$ are given by
\begin{equation} \label{eq3}
\mathcal{E}_G(v)=\frac{1}{\pi}\int_\mathbb{R} 1- \frac{\mathbf{i}x\, Q_j(\mathbf{i}x)}{P(\mathbf{i}x)}dx,
\end{equation}
and 
\begin{equation} \label{eq4}
\mathcal{E} _{\tilde G}(v)=\frac{1}{\pi}\int_\mathbb{R} 1- \frac{\mathbf{i}x\, Q_j(\mathbf{i}x)}{\tilde P(\mathbf{i}x)}dx,
\end{equation}

Since  $P\succeq P$, we obtain the conclusion.
\end{proof}

\subsection{Coverings}

In Arizmendi et al. \cite{AFJ}, it was shown that for a bipartite graph $G$ with parts $G_1$ and $G_2$ the following equalities hold,
\begin{equation}\label{bipartite part}
\frac{1}{2}\mathcal{E}(G)=\sum_{i\in G_1} \mathcal{E}_G(v_i)=\sum_{i\in G_2} \mathcal{E}_G(v_i).
\end{equation}
In this section we prove that for  bipartite graphs the sum of the vertex energies of the cover is at least half of the total energy. 
\begin{theorem}\label{COVERT} Let $C$ be a vertex cover of the graph $G$. Then we have the inequality 
\begin{equation} \label{cover}  \sum_{i\in C} \mathcal{E}_G(v_i) \geq \frac{1}{2}\mathcal{E}(G). \end{equation}
\end{theorem}
The above theorem is a generalization of \eqref{bipartite part} since $G_1$ and $G_2$ are disjoint coverings and thus the inequalities from the equation \eqref{cover} above must actually be equalities.

Before going into the proof, let $\phi(G-v_j;\mathbf{i}x)= \sum_{k\geq 0}  b(j)_{2k} x^{n-2k}$ and  let us 
rewrite our main formula as 
\begin{eqnarray}\label{rewrite}
\mathcal{E}_G(v_j)&=&\frac{1}{\pi}\int_\mathbb{R} 1- \frac{\mathbf{i}x\, \phi(G-v_j;\mathbf{i}x)}{\phi(G;\mathbf{i}x)}dx \nonumber \\
&=&\frac{1}{\pi}\int_\mathbb{R} 1- \frac{ \sum_{k\geq 0}  b(j)_{2k} (x)^{-2k}}{\sum_{k\geq 0} b_{2k} (x)^{-2k}}dx\\
&=&\frac{1}{\pi}\int_\mathbb{R}\frac{\sum_{k\geq 0} b_{2k} (x)^{-2k}   -    \sum_{k\geq 0}  b(j)_{2k} (x)^{-2k}}{\sum_{k\geq 0} b_{2k} (x)^{-2k}}dx \nonumber 
\end{eqnarray}

\begin{proof}[Proof of Theorem \ref{COVERT}]
Our goal is to show that
$$\sum^n_i\varepsilon_G(v_i)\leq 2\sum_{i\in C} \varepsilon_G(v_i).$$

In view of  \eqref{rewrite} it is enough to show that for all $k$
$$\sum_{j}  (b_{2k}-b(j)_{2k})\leq \sum_{j\in C} 2 (b_{2k}-b(j)_{2k}).$$
Now we use the fact that since G is a forest by Sachs theorem we have the follwing
\begin{eqnarray*} b_{2k}&=&|\{\text{$k$-matchings on $G$}\}|, \\
b(j)_{2k}&=&|\{\text{$k$-matchings on $G-v_j$}\}|.
\end{eqnarray*}
Thus, their difference, $\tilde b(j)_2k:=b_{2k}-b(j)_{2k}$, counts the number of $k$-matchings in $G$ with at least one matching containing the vertex $j$. 

  Let us denote by $\tilde E(j)$ the set of of the form $i~j$ and for an edge $e$ $m_k(e)$ the number of matchings containing $e$ 
then $$\sum^n_{j}  (b_{2k}-b(j)_{2k})=\sum_{j}^n\tilde b(j)=\sum_{j}^n \sum_{e\in E(j)} m_k(e)=2\sum_{e\in E} m_k(e)$$
On the other hand since $C$ is a cover we have
$$\sum_{j\in C}   (b_{2k}-b(j)_{2k})=\sum_{j\in C} \tilde b(j)=\sum_{j\in C} \sum_{e\in E(j)} m_k(e)\geq \sum_{e\in E} m_k(e)$$
as desired.

\end{proof}

The following two corollaries are direct consequences of the above.

\begin{corollary}Let $G$ be a bipartite graph and let $C$  be a vertex  cover of $G$ then 
$\frac{1}{2}\mathcal{E}(G)\leq\sum_{j\in C}\sqrt{d_j}.$
\end{corollary}
\begin{proof}
Indeed from \cite{ArJu} we know that $\mathcal{E}_G(v_j)\leq\sqrt{d_j}$, summing over $j\in C$ we obtain the inequality.
\end{proof}

\begin{corollary}Let $G$ be a  bipartite graph and let $I$ be and independent set in $G$ then 
$\frac{1}{2}\mathcal{E}(G)\geq\sum_{i\in I} \mathcal{E}_G(v_i).$
\end{corollary}
\begin{proof}
Since $I$ is an independent set then $V\setminus I$ is a cover. The result follows by Theorem \ref{COVERT}.
\end{proof}

\subsection{Explicit examples}

We show how to use  Theorem \ref{acyclic2} to compare vertex energies in Paths and Cycles. 
The following lemma from \cite{Glibro}, tells us the relationship between the quasi-order on the unions of $P_n$ graphs.

\begin{lemma}
Let $n=4k,\, 4k+1,\, 4k+2,\text{ or } 4k+3$. Then 
$$
P_n\succ P_2\cup P_{n-2} \succ P_4\cup P_{n-4}\succ \cdots \succ P_{2k}\cup P_{n-2k}\succ P_{2k+1}\cup P_{n-2k-1}$$
$$\succ P_{2k-1}\cup P_{n-2k+1}\succ \cdots \succ P_{3}\cup P_{n-3}\succ P_1\cup P_{n-1}.
$$
\end{lemma}

Let the vertices $v_1,v_2,\dots,v_n$ of the path $P_n$ be labelled so that vertices $v_1$ and $v_n$ are leaves and vertices $v_j$ and $v_{j+1}$ are adjacent, with $j=1,2,\dots,n-1$.
Note that $\phi(P_n-v_j;\mathbf{i}x)=\phi(P_{j-1};\mathbf{i}x) \phi(P_{n-j};\mathbf{i}x)$, as a consequence  $\phi(P_n-v_j;\mathbf{i}x)=\phi(P_n-v_{n-j+1};\mathbf{i}x)$.

\begin{example}
To exemplify Lemma 4.3. Consider $P_7$, the path of size $7$. Removing vertices $v_1$ to $v_7$ we get the following union of  paths, with their respective polynomials:
\end{example}

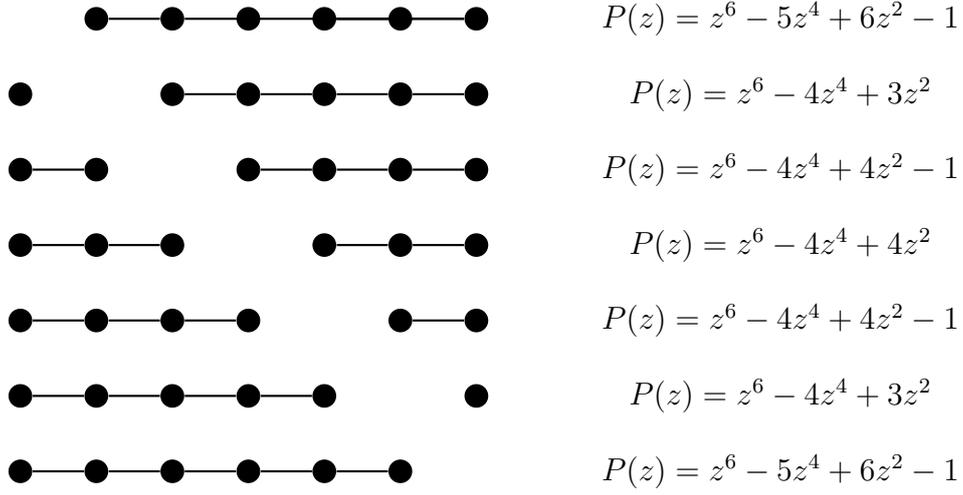
\begin{figure}
[h]\centering 
  \begin{tikzpicture}
\node[style={circle,fill=black,inner sep=2pt]}] (27) at (1,7) {$.$};
\node[style={circle,fill=black,inner sep=2pt]}] (37) at (2,7) {$.$};
\node[style={circle,fill=black,inner sep=2pt]}] (47) at (3,7) {$.$};
\node[style={circle,fill=black,inner sep=2pt]}] (57) at (4,7) {$.$};
\node[style={circle,fill=black,inner sep=2pt]}] (67) at (5,7) {$.$};
\node[style={circle,fill=black,inner sep=2pt]}] (47) at (6,7) {$.$};
\path[draw,thick](27) edge node {} (37);
\path[draw,thick](37) edge node {} (47);
\path[draw,thick](47) edge node {} (57);
\path[draw,thick](57) edge node {} (67);
\node[shift={(10,7)}]{$P(z)=z^6-5z^4+6z^2-1$};

\node[style={circle,fill=black,inner sep=2pt]}] (16) at (0,6) {$.$};
\node[style={circle,fill=black,inner sep=2pt]}] (36) at (2,6) {$.$};
\node[style={circle,fill=black,inner sep=2pt]}] (46) at (3,6) {$.$};
\node[style={circle,fill=black,inner sep=2pt]}] (56) at (4,6) {$.$};
\node[style={circle,fill=black,inner sep=2pt]}] (66) at (5,6) {$.$};
\node[style={circle,fill=black,inner sep=2pt]}] (76) at (6,6) {$.$};
\path[draw,thick](36) edge node {} (46);
\path[draw,thick](46) edge node {} (56);
\path[draw,thick](56) edge node {} (66);
\path[draw,thick](66) edge node {} (76);

 \node[shift={(10,6)}]{$P(z)=z^6-4z^4+3z^2$};

  \node[style={circle,fill=black,inner sep=2pt]}] (15) at (0,5) {$.$};
\node[style={circle,fill=black,inner sep=2pt]}] (25) at (1,5) {$.$};
\node[style={circle,fill=black,inner sep=2pt]}] (45) at (3,5) {$.$};
\node[style={circle,fill=black,inner sep=2pt]}] (55) at (4,5) {$.$};
\node[style={circle,fill=black,inner sep=2pt]}] (65) at (5,5) {$.$};
\node[style={circle,fill=black,inner sep=2pt]}] (75) at (6,5) {$.$};
\path[draw,thick](15) edge node {} (25);
\path[draw,thick](45) edge node {} (55);
\path[draw,thick](55) edge node {} (65);
\path[draw,thick](65) edge node {} (75);
 \node[shift={(10,5)}]{$P(z)=z^6-4z^4+4z^2-1$};

\node[style={circle,fill=black,inner sep=2pt]}] (14) at (0,4) {$.$};
\node[style={circle,fill=black,inner sep=2pt]}] (24) at (1,4) {$.$};
\node[style={circle,fill=black,inner sep=2pt]}] (34) at (2,4) {$.$};
\node[style={circle,fill=black,inner sep=2pt]}] (54) at (4,4) {$.$};
\node[style={circle,fill=black,inner sep=2pt]}] (64) at (5,4) {$.$};
\node[style={circle,fill=black,inner sep=2pt]}] (74) at (6,4) {$.$};
\path[draw,thick](14) edge node {} (24);
\path[draw,thick](24) edge node {} (34);
\path[draw,thick](54) edge node {} (64);
\path[draw,thick](64) edge node {} (74);

 \node[shift={(10,4)}]{$P(z)=z^6-4z^4+4z^2$};

  \node[style={circle,fill=black,inner sep=2pt]}] (13) at (0,3) {$.$};
  \node[style={circle,fill=black,inner sep=2pt]}] (23) at (1,3) {$.$};
\node[style={circle,fill=black,inner sep=2pt]}] (33) at (2,3) {$.$};
\node[style={circle,fill=black,inner sep=2pt]}] (43) at (3,3) {$.$};
\node[style={circle,fill=black,inner sep=2pt]}] (63) at (5,3) {$.$};
\node[style={circle,fill=black,inner sep=2pt]}] (73) at (6,3) {$.$};
\path[draw,thick](23) edge node {} (33);
\path[draw,thick](33) edge node {} (43);
\path[draw,thick](63) edge node {} (73);
\path[draw,thick](13) edge node {} (23);
 \node[shift={(10,3)}]{$P(z)=z^6-4z^4+4z^2-1$};

  \node[style={circle,fill=black,inner sep=2pt]}] (12) at (0,2) {$.$};
  \node[style={circle,fill=black,inner sep=2pt]}] (22) at (1,2) {$.$};
\node[style={circle,fill=black,inner sep=2pt]}] (32) at (2,2) {$.$};
\node[style={circle,fill=black,inner sep=2pt]}] (42) at (3,2) {$.$};
\node[style={circle,fill=black,inner sep=2pt]}] (52) at (4,2) {$.$};
\node[style={circle,fill=black,inner sep=2pt]}] (72) at (6,2) {$.$};
\path[draw,thick](22) edge node {} (32);
\path[draw,thick](32) edge node {} (42);
\path[draw,thick](42) edge node {} (52);
\path[draw,thick](12) edge node {} (22);
 \node[shift={(10,2)}]{$P(z)=z^6-4z^4+3z^2$};

  \node[style={circle,fill=black,inner sep=2pt]}] (11) at (0,1) {$.$};
  \node[style={circle,fill=black,inner sep=2pt]}] (21) at (1,1) {$.$};
\node[style={circle,fill=black,inner sep=2pt]}] (31) at (2,1) {$.$};
\node[style={circle,fill=black,inner sep=2pt]}] (41) at (3,1) {$.$};
\node[style={circle,fill=black,inner sep=2pt]}] (51) at (4,1) {$.$};
\node[style={circle,fill=black,inner sep=2pt]}] (61) at (5,1) {$.$};
\path[draw,thick](21) edge node {} (31);
\path[draw,thick](31) edge node {} (41);
\path[draw,thick](41) edge node {} (51);
\path[draw,thick](51) edge node {} (61);
\path[draw,thick](11) edge node {} (21);
 \node[shift={(10,1)}]{$P(z)=z^6-5z^4+6z^2-1$};
 \end{tikzpicture}
 \caption{Remotion of vertices $v_1$ to $v_7$ in path $P_7$ and their respective polynomials.}
 \end{figure}

From the previous lemma, and Lemma \ref{acyclic} we have the next corollary.

\begin{corollary}
$$\mathcal{E}(v_1)<\mathcal{E}(v_3)<\mathcal{E}(v_5)< \cdots <\mathcal{E}(v_{2k+1})< \mathcal{E}(v_{2k+2})< \mathcal{E}(v_{2k})< \cdots <\mathcal{E}(v_{4}) < \mathcal{E}(v_2)$$
\end{corollary}

{\bf Trees}\\
Consider the notation $i(v)j$ for trees, as in \cite{Glibro}, where $i$ is the number of vertices of a path $P_i$, such that in vertex $(v)$ (position from left to right) has a  path $P_j$ with $j$ vertices, for instance $5(4)3$ corresponds to Figure \ref{arbol}.

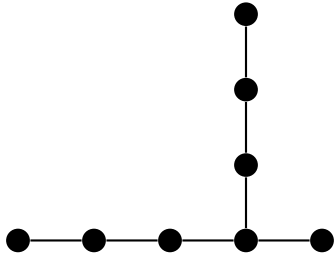
\begin{figure}[h]\centering 
  \begin{tikzpicture}\label{arbol}
\node[style={circle,fill=black,inner sep=2pt]}] (1) at (1,0) {$.$};
\node[style={circle,fill=black,inner sep=2pt]}] (2) at (2,0) {$.$};
\node[style={circle,fill=black,inner sep=2pt]}] (3) at (3,0) {$.$};
\node[style={circle,fill=black,inner sep=2pt]}] (4) at (4,0) {$.$};
\node[style={circle,fill=black,inner sep=2pt]}] (5) at (5,0) {$.$};
\node[style={circle,fill=black,inner sep=2pt]}] (41) at (4,1) {$.$};
\node[style={circle,fill=black,inner sep=2pt]}] (42) at (4,2) {$.$};
\node[style={circle,fill=black,inner sep=2pt]}] (43) at (4,3) {$.$};
\path[draw,thick](1) edge node {} (2);
\path[draw,thick](2) edge node {} (3);
\path[draw,thick](3) edge node {} (4);
\path[draw,thick](4) edge node {} (41);
\path[draw,thick](4) edge node {} (5);
\path[draw,thick](41) edge node {} (42);
\path[draw,thick](42) edge node {} (43);
 \end{tikzpicture}
 \caption{$5(4)3$ tree.}
 \end{figure}

\begin{lemma}
Let $t$ be an integer and $k=\lfloor {n/2} \rfloor$. Then,

\begin{eqnarray*}
n-1(2)1 \prec n-1(4)1 &\prec&  \dots \prec n-1(k)1 \prec n-1(k-1)1 \prec n-1(k-3)1\prec\cdots \\ \dots &\prec& n-1(3)1  \prec n-1(1)1\equiv P_n \,\,  \qquad  \text{if}\,\, n=4t
\end{eqnarray*}

\begin{eqnarray*}
n-1(2)1 \prec n-1(4)1 &\prec&  \dots n-1(k-1)1 \prec n-1(k)1 \prec n-1(k-2)1 \prec\dots \\ \dots &\prec& n-1(3)1   \prec n-1(1)1\equiv P_n \,\, \qquad  \text{if}\,\, n=4t+2 
\end{eqnarray*}

\begin{eqnarray*}
n-1(2)1 \prec n-1(4)1 &\prec & \dots n-1(k)1 \prec n-1(k+1)1 \prec n-1(k-1)1  \prec\dots \\ \dots &\prec& n-1(3)1  \prec n-1(1)1\equiv P_n \,\, \qquad \text{if}\,\, n=4t+1 
\end{eqnarray*}

\begin{eqnarray*}
n-1(2)1 \prec n-1(4)1 &\prec&  \dots n-1(k+1)1 \prec n-1(k)1 \prec n-1(k-2)1  \prec\dots \\ \dots &\prec&  n-1(3)1   \prec n-1(1)1\equiv P_n \,\,  \qquad  \text{if}\,\, n=4t+3 
\end{eqnarray*}
\end{lemma}

Consider a path with $n-1$ vertices, if we add a vertex in position $i$ with $i=1 \dots n-1$, obtaining the tree $n-1(i)1$ the energy of the aggregated vertex depends on the position where we added it. Since we are comparing the vertex $n$ in every tree we will use the notation $v_{(i)}$ to denote the position where the vertex was added to $P_{n-1}$. From the previous lemma,  we have the next corollary.
\begin{corollary} 
Let $P_{n-1}$ and $\mathcal{E}(v_{(i)})$ the energy of the vertex added in position $i$ to $P_{n-1}$ in order to  construct a tree with $n$ vertices. $\mathcal{E}(v_{(i)})$ depends on the value of $i$ in the following form

\begin{eqnarray*}
\mathcal{E}(v_{(2)}) <\mathcal{E}(v_{(4)})<\cdots <\mathcal{E}(v_{(k)})< \mathcal{E}(v_{(k-1)})< \mathcal{E}(v_{(k-3)})< \cdots <\mathcal{E}(v_{(3)}) < \mathcal{E}(v_{(1)}) \,\, \text{if}\,\, n=4t
\end{eqnarray*}
\begin{eqnarray*}
\mathcal{E}(v_{(2)}) <\mathcal{E}(v_{(4)})<\cdots <\mathcal{E}(v_{(k-1)})< \mathcal{E}(v_{(k)})< \mathcal{E}(v_{(k-2)})< \cdots <\mathcal{E}(v_{(3)}) < \mathcal{E}(v_{(1)}) \,\, \text{if}\,\, n=4t+2
\end{eqnarray*}
\begin{eqnarray*}
\mathcal{E}(v_{(2)}) <\mathcal{E}(v_{(4)})<\cdots <\mathcal{E}(v_{(k)})< \mathcal{E}(v_{(k+1)})< \mathcal{E}(v_{(k-1)})< \cdots <\mathcal{E}(v_{(3)}) < \mathcal{E}(v_{(1)}) \,\, \text{if}\,\, n=4t+1
\end{eqnarray*}
\begin{eqnarray*}
\mathcal{E}(v_{(2)}) <\mathcal{E}(v_{(4)})<\cdots <\mathcal{E}(v_{(k+1)})< \mathcal{E}(v_{(k)})< \mathcal{E}(v_{(k-2)})< \cdots <\mathcal{E}(v_{(3)}) < \mathcal{E}(v_{(1)}) \,\, \text{if}\,\, n=4t+3
\end{eqnarray*}

\end{corollary}

On the other hand, we can compare some vertices in cycles and paths.

\begin{example}
Now we may also compare  $\mathcal{E}(w_1)$ from $P_n$ with $\mathcal{E}(w_j)$ for any vertex $v_j$ on $C_n$, such that $n=4k+2$ we get that $\mathcal{E}(v_j)>\mathcal{E}(w_1)$, since from Lemma \ref{acyclic2} we consider $G=C_n$ and $H=P_n$, and $G \cup (H-w)= C_n \cup P_{n-1}  \succeq  H \cup (G-v) = P_n \cup P_{n-1} $, it is known $ C_n   \succeq P_n $ for $n=4k+2$, since we can write $C_n=\phi(P_n)-\phi(P_{n-2})-2$.
\end{example}

\section{Conclusions}
In this work we provided  a Coulson integral formula for the vertex energy of a graph, this formula extends our comprehension of energy of graphs by means of a better understanding of the vertex energy. The shown examples and applications allow us to understand the interaction between the individual and global energy in some bipartite graphs, by considering the quasi-order on these kinds of graphs.  We believe that future work could provide a combinatorial interpretation of the Coulson integral formula for a vertex in order to deeply understand the graph energy from several points of view.
 
 \section{Acknowledgment}

  The authors thank Jes\'us Emilio Dom\'inguez Russell for useful comments during the preparation of this paper. O. Arizmendi was supported by a CONACYT Grant No. 222668 and by the  European Union's Horizon 2020 research and innovation programme under the Marie Sk\l{}odowska-Curie grant agreement 734922, during the writing of this paper.

\end{document}